\documentclass{amsart}
\usepackage{amssymb}
\usepackage{mathrsfs}
\usepackage{cases}
\usepackage{amsmath}
\usepackage{amsfonts}
\usepackage{pifont}
\usepackage{graphicx,amssymb,mathrsfs,amsmath}
\usepackage{tocvsec2}
\newtheorem{thm}{Theorem}[section]
\newtheorem{cor}[thm]{Corollary}
\newtheorem{lem}[thm]{Lemma}
\newtheorem{prop}[thm]{Proposition}
\theoremstyle{definition}
\newtheorem{defn}[thm]{Definition}
\newtheorem{example}[thm]{Example}
\theoremstyle{remark}
\newtheorem{rem}[thm]{Remark}
\numberwithin{equation}{section}

\begin{document}

\title{The existence and uniqueness of almost periodic and asymptotically almost periodic solutions of semilinear Cauchy inclusions}

\author{Marko Kosti\' c}
\address{Faculty of Technical Sciences,
University of Novi Sad,
Trg D. Obradovi\' ca 6, 21125 Novi Sad, Serbia}
\email{marco.s@verat.net}

\begin{abstract}
The main aim of this paper is to investigate almost periodicity and asymptotic almost periodicity
of abstract semilinear Cauchy inclusions of first order with (asymptotically) Stepanov almost periodic coefficients. To achieve our goal, we employ fixed point theorems and the well known results on the generation of infinitely differentiable degenerate semigroups with removable singularites at zero. 
\end{abstract}

{\renewcommand{\thefootnote}{} \footnote{2010 {\it Mathematics
Subject Classification. 34G25, 47D03, 47D06, 47D99.
\\ \text{  }  \ \    {\it Key words and phrases.} Abstract semilinear Cauchy inclusions, Multivalued linear operators,  Almost periodicity, Asymptotic almost periodicity, Stepanov asymptotic almost periodicity.
\\  \text{  } The author is partially supported by grant 174024 of Ministry
of Science and Technological Development, Republic of Serbia.}}

\maketitle

\section{Introduction and preliminaries}

Almost periodic and asymptotically almost periodic solutions of differential equations in Banach spaces have been considered by many authors so far (for the basic information on the subject, we refer the reader to the monographs by D. N. Cheban \cite{cheban} and Y. Hino, T. Naito, N. V. Minh, J. S. Shin \cite{hino-bor}).
In the paper under review, we continue our recent research studies \cite{AOT}-\cite{EJDE} by enquiring into the existence of a unique almost periodic solution or a unique asymptotically almost periodic solution for a class
of abstract semilinear Cauchy inclusions of first order with (asymptotically) Stepanov almost periodic coefficients.
For this purpose, we introduce the class of asymptotically Stepanov almost periodic functions depending on two parameters and prove some new composition principles in this direction (see e.g. \cite{superposition}, \cite{comp-adv} and references therein). It seems that our main results, Theorem \ref{zeljezo}-Theorem \ref{stepa-frimex}, are new even for abstract semilinear non-degenerate differential equations with almost sectorial operators (\cite{pb1}-\cite{periago}).

The organization and main ideas of this paper can be briefly described as follows. In Proposition \ref{sew}, we reconsider the notion of an asymptotically almost periodic function depending on two parameters, while in Definition \ref{lot} we introduce the class of asymptotically Stepanov almost periodic two-parameter functions. A useful characterization
of this class is proved  in Lemma \ref{tricky-prim} following the ideas of W. M. Ruess, W. H. Summers \cite{RUESS} and  H. R. Henr\' iquez \cite{hernan1}. We open the second section of paper by proving some new composition principles for Stepanov almost periodic two-parameter functions and asymptotically Stepanov almost periodic two-parameter functions.
The main aim of Theorem \ref{vcb} is to clarify that the composition principle \cite[Theorem 2.2]{comp-adv}, proved by W. Long and S.-H. Ding, continues to hold for the functions defined on the real semi-axis $I =[0,\infty).$ The use of usual Lipschitz assumption has some advantages compared to the condition $f \in {\mathcal L}^{r}({\mathbb R} \times X : X)$ used in the formulation of the above-mentioned theorem since, in this case, we can include the order of (asymptotic) Stepanov almost periodicity $p =1$ in our analyses  (cf. Theorem \ref{vcb-prim} for more details). In Proposition \ref{bibl}-Proposition \ref{bibl-frimaonji}, we analyze composition principles for asymptotically Stepanov almost periodic two-parameter functions. The main aim of Lemma \ref{andrade} is to prove that the function defined through the infinite convolution product \eqref{wer} is asymptotically almost periodic provided that the operator family in its definition is exponentially decaying at infinity and has a removable singularity at zero, as well as that the coefficient $f(\cdot) $ is asymptotically Stepanov almost periodic. In the remaining part of paper, we examine the class of multivalued linear operators ${\mathcal A}$ satisfying the condition \cite[(P), p. 47]{faviniyagi} introduced by A. Favini and A. Yagi:
\begin{itemize} \index{removable singularity at zero}
\item[(P)]
There exist finite constants $c,\ M>0$ and $\beta \in (0,1]$ such that\index{condition!(PW)}
$$
\Psi:=\Psi_{c}:=\Bigl\{ \lambda \in {\mathbb C} : \Re \lambda \geq -c\bigl( |\Im \lambda| +1 \bigr) \Bigr\} \subseteq \rho({\mathcal A})
$$
and
$$
\| R(\lambda : {\mathcal A})\| \leq M\bigl( 1+|\lambda|\bigr)^{-\beta},\quad \lambda \in \Psi.
$$
\end{itemize}
The main goal of Theorem \ref{zeljezo}-Theorem \ref{zeljeznica} is to prove the existence of a unique almost periodic mild solution of the following semilinear differential inclusion of first order
\begin{align}\label{favini}
u^{\prime}(t)\in {\mathcal A}u(t)+f(t,u(t)),\quad t\in {\mathbb R},
\end{align}
where $f : {\mathbb R} \times X \rightarrow X$ is Stepanov almost periodic and some extra conditions are satisfied. Also, of concern is the following semilinear Cauchy inclusion of first order
\[
\hbox{(DFP)}_{f,s} : \left\{
\begin{array}{l}
u^{\prime}(t)\in {\mathcal A}u(t)+f(t,u(t)),\ t\geq 0,\\
\quad u(0)=u_{0}.
\end{array}
\right.
\]
In Theorem \ref{stepa}-Theorem \ref{stepa-frimex}, we analyze the existence of a unique asymptotically almost periodic solution of semilinear differential inclusion $(DFP)_{f,s}$ provided that the coefficient $f(\cdot,\cdot)$ behaves asymptotically in time as a Stepanov almost periodic function. Some simple consequences of Theorem \ref{stepa-frimex} are stated in Corollary \ref{kretinjo} and Corollary \ref{kret}. The main purpose of Remark \ref{brazilci}(i) is to explain how we can use the established results of ours with a view to prove a slight extension of  \cite[Theorem 4.4]{brazil}, one of the main results of investigation \cite{brazil} carried out by B. de Andrade and C. Lizama.
In Example \ref{merkle-nedegenerisane}, we present some applications to the abstract higher-order semilinear differential equations in  H\"older spaces, while in Example \ref{markec} we analyze the existence of a unique (asymptotically) almost periodic solution for semilinear Poisson heat equations in $L^{p}$-spaces.
The analysis of existence and uniqueness of pseudo-almost periodic solutions for a class of fractional Sobolev inclusions will be considered in our forthcoming paper \cite{klavir} (see \cite{ding}, \cite{fatajou} and \cite{lee} for some researches about Stepanov-like almost automorphic
solutions of abstract differential equations). 

We use the standard notation throughout the paper.
By $X$ we denote a complex Banach space. If $Y$ is also such a space, then by
$L(X,Y)$ we denote the space of all continuous linear mappings from $X$ into
$Y;$ $L(X)\equiv L(X,X).$ If $A$ is a linear operator
acting on $X,$
then the domain, kernel space and range of $A$ will be denoted by
$D(A),$ $N(A)$ and $R(A),$
respectively. 
By $C_{b}([0,\infty):X)$ we denote the space consisted of all bounded continuous functions from $[0,\infty)$ into $X;$ the symbol $C_{0}([0,\infty):X)$ denotes the closed subspace of $C_{b}([0,\infty):X)$ consisting of functions vanishing at infinity.  By $BUC([0,\infty):X)$ we denote the space consisted of all bounded uniformly continuous functions from $[0,\infty)$
to $X.$ This space becomes one of Banach's when equipped with the sup-norm. 

Given $s\in {\mathbb R}$ in advance, set $\lfloor s \rfloor :=\sup \{
l\in {\mathbb Z} : s\geq l \}$ and $\lceil s \rceil:=\inf \{
l\in {\mathbb Z} : s\leq l \}.$ The Gamma function is denoted by
$\Gamma(\cdot)$ and the principal branch is always used to take
the powers.

As it is well known, the notion of an almost periodic function was introduced by H. Bohr in 1925 and later generalized by many other mathematicians (cf. \cite{diagana}, \cite{gaston} and \cite{18} for more details on the subject).
Let $I={\mathbb R}$ or $I=[0,\infty),$ and let $f : I \rightarrow X$ be continuous. Given $\epsilon>0,$ we call $\tau>0$ an $\epsilon$-period for $f(\cdot)$ iff
$
\| f(t+\tau)-f(t) \| \leq \epsilon, $ $ t\in I.
$
The set constituted of all $\epsilon$-periods for $f(\cdot)$ is denoted by $\vartheta(f,\epsilon).$ It is said that $f(\cdot)$ is almost periodic, a.p. for short, iff for each $\epsilon>0$ the set $\vartheta(f,\epsilon)$ is relatively dense in $I,$ which means that
there exists $l>0$ such that any subinterval of $I$ of length $l$ meets $\vartheta(f,\epsilon)$. The space consisted of all almost periodic functions from the interval $I$ into $X$ will be denoted by $AP(I:X).$

The class of asymptotically almost periodic functions was introduced by M. Fr\' echet in 1941 (for more details concerning the vector-valued asymptotically almost periodic functions, see e.g. \cite{cheban},  \cite{diagana} and \cite{gaston}). A function $f \in C_{b}([0,\infty) : X)$ is said to be asymptotically almost periodic iff
for every $\epsilon >0$ we can find numbers $ l > 0$ and $M >0$ such that every subinterval of $[0,\infty)$ of
length $l$ contains, at least, one number $\tau$ such that $\|f(t+\tau)-f(t)\| \leq \epsilon$ for all $t \geq M.$
The space consisted of all asymptotically almost periodic functions from $[0,\infty)$ into $X$ will be denoted by
$AAP([0,\infty) : X).$ It is well known that for any function $f \in C([0,\infty):X),$ the following
statements are equivalent:
\begin{itemize}
\item[(i)] $f\in AAP([0,\infty) :X).$
\item[(ii)] There exist uniquely determined functions $g \in AP([0,\infty) :X)$ and $\phi \in  C_{0}([0,\infty): X)$
such that $f = g+\phi.$
\item[(iii)] The set $H(f):=\{f(\cdot +s) : s\geq 0\}$ is relatively compact in $C_{b}([0,\infty):X).$
\end{itemize}

Let $1\leq p <\infty.$ Then we say that a function $f\in L^{p}_{loc}(I :X)$ is Stepanov $p$-bounded, $S^{p}$-bounded shortly, iff
$$
\|f\|_{S^{p}}:=\sup_{t\in I}\Biggl( \int^{t+1}_{t}\|f(s)\|^{p}\, ds\Biggr)^{1/p}<\infty.
$$
The space $L_{S}^{p}(I:X)$ consisted of all $S^{p}$-bounded functions becomes a Banach space when equipped with the above norm. 
A function $f\in L_{S}^{p}(I:X)$ is said to be Stepanov $p$-almost periodic, $S^{p}$-almost periodic shortly, iff the function
$
\hat{f} : I \rightarrow L^{p}([0,1] :X),
$ defined by 
$$
\hat{f}(t)(s):=f(t+s),\quad t\in I,\ s\in [0,1]
$$
is almost periodic (cf. M. Amerio, G. Prouse \cite{amerio} for more details). 
It is said that $f\in  L_{S}^{p}([0,\infty): X)$ is asymptotically Stepanov $p$-almost periodic, asymptotically $S^{p}$-almost periodic shortly, iff $\hat{f} : [0,\infty) \rightarrow L^{p}([0,1]:X)$ is asymptotically almost periodic. By $APS^{p}([0,\infty) : X)$ and $AAPS^{p}([0,\infty) : X)$ we denote the classes consisting of all Stepanov $p$-almost periodic functions and asymptotically Stepanov $p$-almost periodic functions, respectively.

It is a well-known fact that if $f(\cdot)$ is an almost periodic (respectively, a.a.p.) function
then $f(\cdot)$ is also $S^p$-almost periodic (resp., $S^p$-a.a.p.) for $1\leq p <\infty.$ The converse statement is false, however.  

We need the assertion of \cite[Lemma 1]{hernan1}:

\begin{lem}\label{tricky}
Suppose that $ f: [0,\infty)\rightarrow X$ is an asymptotically $S^p$-almost periodic
function. Then there are two locally $p$-integrable functions $g: {\mathbb R} \rightarrow X$ and
$q: [0,\infty)\rightarrow X$ satisfying the following conditions:
\begin{itemize}
\item[(i)] $g$ is $S^p$-almost periodic,
\item[(ii)] $\hat{q}$ belongs to the class $C_{0}([0,1] : L^{p}([0,1]:X)),$
\item[(iii)] $f(t)=g(t)+q(t)$ for all $t\geq 0.$
\end{itemize}
Moreover, there exists an increasing sequence $(t_{n})_{n\in {\mathbb N}}$ of positive reals such that $\lim_{n\rightarrow \infty}t_{n}=\infty$
and $g(t)=lim_{n\rightarrow \infty}f(t+ t_{n})$ a.e. $t\geq 0.$
\end{lem}

By $C_{0}([0,\infty) \times Y : X)$ we denote the space of all
continuous functions $h : [0,\infty)  \times Y \rightarrow X$ such that $\lim_{t\rightarrow \infty}h(t, y) = 0$ uniformly for $y$ in any compact subset of $Y .$ A continuous function $f : I  \times Y  \rightarrow X$ is called uniformly continuous on
bounded sets, uniformly for $t \in I$ iff for every $	\epsilon > 0$ and every bounded subset $K$ of $Y$ there
exists a number $\delta_{\epsilon,K }> 0$ such that $\|f(t , x)- f (t, y)\| \leq \epsilon$ for all $ t \in I$ and all $x,\ y \in K$ satisfying that $\|x-y\|\leq \delta_{\epsilon,K }.$ If $f : I  \times Y  \rightarrow X,$ then we define $\hat{f} : I  \times Y  \rightarrow L^{p}([0,1]:X)$ by $\hat{f}(t , y):=f(t +\cdot, y),$ $t\geq 0,$ $y\in Y.$

For the purpose of research of (asymptotically) almost periodic properties of solutions to semilinear Cauchy inclusions, we need to remind ourselves of the following well-known definitions and results (see e.g. C. Zhang \cite{zhang-c-prim}, W. Long, S.-H. Ding \cite{comp-adv}, and Proposition \ref{sew} below):

\begin{defn}\label{definicija}
Let $1\leq p <\infty.$ 
\begin{itemize}
\item[(i)]
A function $f : I \times Y \rightarrow X$ is called almost periodic iff $f (\cdot, \cdot)$ is bounded, continuous as well as for every $\epsilon>0$ and every compact
$K\subseteq Y$ there exists $l(\epsilon,K) > 0$ such that every subinterval $J\subseteq I$ of length $l(\epsilon,K)$ contains a number $\tau$ with the property that $\|f (t +\tau , y)- f (t, y)\| \leq \epsilon$ for all $t \in  I,$ $ y \in K.$ The collection of such functions will be denoted by $AP(I \times Y : X).$
\item[(ii)] A function $f : [0,\infty)  \times Y \rightarrow X$ is said to be asymptotically almost periodic iff it is bounded continuous and admits a
decomposition $f = g + q,$ where $g \in AP([0,\infty)  \times Y : X)$ and $q\in C_{0}([0,\infty)  \times Y : X).$ Denote by
 $AAP([0,\infty)  \times Y : X) $ the vector space consisting of all such functions.
\item[(iii)] A function $f : I \times Y \rightarrow X$ is called Stepanov $p$-almost periodic, $S^{p}$-almost periodic shortly, iff $\hat{f} : I  \times Y  \rightarrow L^{p}([0,1]:X)$ is almost periodic.
\end{itemize}
\end{defn}

\begin{lem}\label{suljpa}
\begin{itemize}
\item[(i)]
Let $f \in AP(I \times Y : X)$ and $h \in AP(I : Y ).$ Then the mapping $t\mapsto f (t,h(t)),$ $t\in I$  belongs to the space $ AP(I:X).$ 
\item[(ii)] Let $f \in AAP([0,\infty) \times Y : X)$ and $h \in AAP([0,\infty) : Y ).$ Then the mapping $t\mapsto f (t,h(t)),$ $t\geq 0$ belongs to the space $AAP([0,\infty) : X).$
\end{itemize}
\end{lem}

In Definition \ref{definicija}(ii), a great number of authors assumes a priori that $g \in AP({\mathbb R} \times Y : X).$ This is slightly redundant on account of the following proposition:

\begin{prop}\label{sew}
Let $f : [0,\infty) \times Y \rightarrow X,$ and let $S\subseteq Y.$ Suppose that, for every $\epsilon>0,$ there exists $l(\epsilon,S) > 0$ such that every subinterval $J\subseteq [0,\infty)$ of length $l(\epsilon,S)$ contains a number $\tau$ with the property that $\|f (t +\tau , y)- f (t, y)\| \leq \epsilon$ for all $t \geq 0,$ $ y \in S$ (this, in particular, holds provided that $f\in AP(I \times Y : X)$). 
Denote by $F (t, y)$ the unique almost periodic extension of function $f (t, y)$ from the interval $[0,\infty)$ to the whole real line, for fixed $ y \in S$
(cf. \cite[Proposition 4.7.1]{a43}).
Then, for every $\epsilon>0,$ with the same $l(\epsilon,S) > 0$ chosen as above, we have that every subinterval $J\subseteq {\mathbb R}$ of length $l(\epsilon,S)$ contains a number $\tau$ with the property that $\|F (t +\tau , y)- F (t, y)\| \leq \epsilon$ for all $t \in {\mathbb R},$ $ y \in S.$
\end{prop}

\begin{proof}
Let $\epsilon>0$ be given in advance, let $l(\epsilon,S) > 0$ be as above, and let $J=[a,b]\subseteq {\mathbb R}$. The assertion is clear provided that $a\geq 0.$ Suppose now that $a<0;$ then we choose a number $\tau_{0}>0$ arbitrarily. Then there exists $\tau'\in J=[\tau_{0},\tau_{0}+b-a]\subseteq  [0,\infty)$ such that $\|f (t +\tau' , y)- f (t, y)\| \leq \epsilon$ for all $t \geq 0,$ $ y \in S.$ Since $\tau:=\tau'-\tau_{0}-|a|\in J,$ it suffices to show that $\|F (t +\tau , y)- F(t, y)\| \leq \epsilon$ for all $t \in {\mathbb R},$ $ y \in S.$
Towards this end, fix a number $t \in {\mathbb R}$ and an element $y \in S.$ Since the mapping $s\mapsto F(s+\tau'-\tau_{0}-|a|, y)-F(s-\tau_{0}-|a|, y), $ $s\in {\mathbb R}$ 
is almost periodic, the equation \cite[(4.24)]{a43} shows that
\begin{align*}
\bigl\| F(t& +\tau'-\tau_{0}-|a|, y)-F(t-\tau_{0}-|a|, y)\bigr\| 
\\ &\leq \bigl\| F(\cdot+\tau'-\tau_{0}-|a|, y)-F(\cdot-\tau_{0}-|a|, y)\bigr\|_{\infty}
\\ & =\sup_{s\geq \tau_{0}+|a|}\bigl\| F(s+\tau'-\tau_{0}-|a|, y)-F(s-\tau_{0}-|a|, y)\bigr\|
\\  & = \sup_{s\geq \tau_{0}+|a|}\bigl\| f(s+\tau'-\tau_{0}-|a|, y)-f(s-\tau_{0}-|a|, y)\bigr\| 
\\ & = \sup_{s\geq 0}\bigl\| f(s+\tau', y)-f(s, y)\bigr\| \leq \epsilon.
\end{align*}
This ends the proof of proposition.
\end{proof}

It is very simple to deduce Lemma \ref{suljpa}(i) with $I =[0,\infty)$ by using  Proposition \ref{sew} and the corresponding result in the case that $I ={\mathbb R}$ (see e.g. \cite[Lemma 2.6]{brazil}). Definition \ref{definicija}(iii) seems to be new for $I =[0,\infty),$ and slightly different from the corresponding notion introduced in \cite[Definition 1.8]{comp-adv}, given in the case that $I ={\mathbb R}.$
Observe also that we automatically assume the boundedness of function $f (\cdot, \cdot)$ in the parts (i) and (ii) of Definition \ref{definicija}, following the approach used in \cite{zhang-c-prim}.

By \cite[Theorem 2.6]{zhang-c-prim}, we have that a bounded continuous function $f : [0,\infty)  \times Y \rightarrow X$ is  asymptotically almost periodic iff for every $\epsilon>0$ and every compact
$K\subseteq Y$ there exist $l(\epsilon,K) > 0$ and $M(\epsilon,K) > 0$  such that every subinterval $J\subseteq [0,\infty)$ of length $l(\epsilon,K)$ contains a number $\tau$ with the property that $\| f (t +\tau , y)- f (t, y)\| \leq \epsilon$ for all $t \geq  M(\epsilon,K),$ $ y \in K.$ We
introduce the notion of an asymptotically Stepanov $p$-almost periodic function $f (\cdot, \cdot)$ as follows:

\begin{defn}\label{lot}
Let $1\leq p <\infty.$ A function $f : [0,\infty)  \times Y \rightarrow X$
is said to be asymptotically $S^p$-almost periodic
iff $\hat{f}: [0,\infty)  \times Y \rightarrow  L^{p}([0,1]:X)$ is asymptotically almost periodic. The collection of such functions will be denoted by $AAPS^{p}([0,\infty) \times Y : X).$
\end{defn}

It is very elementary to prove that any asymptotically almost periodic function is also asymptotically Stepanov $p$-almost periodic
($1\leq p <\infty$). Now we state the following two-variable analogue of Lemma \ref{tricky}:

\begin{lem}\label{tricky-prim}
Suppose that $ f: [0,\infty) \times Y \rightarrow X$ is an asymptotically $S^p$-almost periodic
function. Then there are two functions $g: {\mathbb R} \times Y \rightarrow X$ and
$q: [0,\infty) \times Y \rightarrow X$ satisfying that for each $y\in Y$ the functions
$g(\cdot,y)$ and 
$q(\cdot,y)$ are locally $p$-integrable, as well as that
the following holds:
\begin{itemize}
\item[(i)] $\hat{g} : {\mathbb R} \times Y \rightarrow L^{p}([0,1]:X)$ is almost periodic,
\item[(ii)] $\hat{q} \in C_{0}([0,\infty) \times Y : L^{p}([0,1]:X)),$
\item[(iii)] $f(t,y)=g(t,y)+q(t,y)$ for all $t\geq 0$ and $y\in Y.$
\end{itemize}
Moreover, for every compact set $K \subseteq Y,$ there exists an increasing sequence $(t_{n})_{n\in {\mathbb N}}$ of positive reals such that $\lim_{n\rightarrow \infty}t_{n}=\infty$
and $g(t,y)=lim_{n\rightarrow \infty}f(t+ t_{n},y)$ for all $ y\in Y$  and a.e. $t\geq 0.$
\end{lem}

\begin{proof}
By the foregoing, we have that $\hat{f} : [0,\infty)  \times Y \rightarrow X$ is bounded continuous and admits a
decomposition $\hat{f }= G + Q,$ where $G \in AP([0,\infty)  \times Y : L^{p}([0,1]:X))$ and $Q\in C_{0}([0,\infty)  \times Y : L^{p}([0,1]:X)).$ Moreover, the proof of \cite[Theorem 2.6]{zhang-c-prim} shows that, for every compact set $K \subseteq Y,$ there exists an increasing sequence $(t_{n})_{n\in {\mathbb N}}$ of positive reals such that $\lim_{n\rightarrow \infty}t_{n}=\infty$
and $G(t,y)=\lim_{n\rightarrow \infty}\hat{f}(t+ t_{n},y)$ for all $ y\in Y$  and $t\geq 0.$ The remaining part of proof follows by applying Lemma \ref{tricky} to the function $\hat{f }(\cdot,y),$ for fixed element $y\in Y,$ and the uniqueness of decomposition $g(\cdot)+q(\cdot)$ in this lemma.
\end{proof}

For the theory of abstract degenerate differential equations, we refer the reader to the monographs by
R. W. Carroll, R. W. Showalter \cite{carol}, A. Favini, A. Yagi \cite{faviniyagi}, 
I. V. Melnikova, A. I. Filinkov \cite{me152} and M. Kosti\'c \cite{FKP}.
In what follows, we will present a brief overview of definitions from the theory of multivalued linear operators in Banach spaces. 

Suppose that $X$ and $Y$ are Banach spaces.
Let us recall that a multivalued map (multimap) ${\mathcal A} : X \rightarrow P(Y)$ is said to be a multivalued
linear operator (MLO) iff the following holds:
\begin{itemize}
\item[(i)] $D({\mathcal A}) := \{x \in X : {\mathcal A}x \neq \emptyset\}$ is a linear subspace of $X$;
\item[(ii)] ${\mathcal A}x +{\mathcal A}y \subseteq {\mathcal A}(x + y),$ $x,\ y \in D({\mathcal A})$
and $\lambda {\mathcal A}x \subseteq {\mathcal A}(\lambda x),$ $\lambda \in {\mathbb C},$ $x \in D({\mathcal A}).$
\end{itemize}
If $X=Y,$ then we say that ${\mathcal A}$ is an MLO in $X.$

The fundamental equality used below says that, if $x,\ y\in D({\mathcal A})$ and $\lambda,\ \eta \in {\mathbb C}$ with $|\lambda| + |\eta| \neq 0,$ then 
$\lambda {\mathcal A}x + \eta {\mathcal A}y = {\mathcal A}(\lambda x + \eta y).$ Assuming ${\mathcal A}$ is an MLO, then ${\mathcal A}0$ is a linear submanifold of $Y$
and ${\mathcal A}x = f + {\mathcal A}0$ for any $x \in D({\mathcal A})$ and $f \in {\mathcal A}x.$ Set $R({\mathcal A}):=\{{\mathcal A}x :  x\in D({\mathcal A})\}.$
Then the set ${\mathcal A}^{-1}0 = \{x \in D({\mathcal A}) : 0 \in {\mathcal A}x\}$ is called the kernel
of ${\mathcal A}$ and it is denoted by either $N({\mathcal A})$ or Kern$({\mathcal A}).$ The inverse ${\mathcal A}^{-1}$ of an MLO is defined by
$D({\mathcal A}^{-1}) := R({\mathcal A})$ and ${\mathcal A}^{-1} y := \{x \in D({\mathcal A}) : y \in {\mathcal A}x\}$.
It can be simply checked that ${\mathcal A}^{-1}$ is an MLO in $X,$ as well as that $N({\mathcal A}^{-1}) = {\mathcal A}0$
and $({\mathcal A}^{-1})^{-1}={\mathcal A};$ ${\mathcal A}$ is said to be injective iff ${\mathcal A}^{-1}$ is
single-valued. 

For any mapping ${\mathcal A}: X \rightarrow P(Y)$ we define $\check{{\mathcal A}}:=\{(x,y) : x\in D({\mathcal A}),\ y\in {\mathcal A}x\}.$ Then ${\mathcal A}$ is an MLO iff $\check{{\mathcal A}}$ is a linear relation in $X\times Y,$ i.e., iff $\check{{\mathcal A}}$ is a linear subspace of $X \times Y.$ 

Assume that ${\mathcal A},\ {\mathcal B} : X \rightarrow P(Y)$ are two MLOs. Then we define its sum ${\mathcal A}+{\mathcal B}$ by $D({\mathcal A}+{\mathcal B}) := D({\mathcal A})\cap D({\mathcal B})$ and $({\mathcal A}+{\mathcal B})x := {\mathcal A}x +{\mathcal B}x,$ $x\in D({\mathcal A}+{\mathcal B}).$
It is clear that ${\mathcal A}+{\mathcal B}$ is likewise an MLO.

Let ${\mathcal A} : X \rightarrow P(Y)$ and ${\mathcal B} : Y\rightarrow P(Z)$ be two MLOs, where $Z$ is an SCLCS. The product of operators ${\mathcal A}$
and ${\mathcal B}$ is defined by $D({\mathcal B}{\mathcal A}) :=\{x \in D({\mathcal A}) : D({\mathcal B})\cap {\mathcal A}x \neq \emptyset\}$ and
${\mathcal B}{\mathcal A}x:=
{\mathcal B}(D({\mathcal B})\cap {\mathcal A}x).$ Then ${\mathcal B}{\mathcal A} : X\rightarrow P(Z)$ is an MLO and
$({\mathcal B}{\mathcal A})^{-1} = {\mathcal A}^{-1}{\mathcal B}^{-1}.$ The scalar multiplication of an MLO ${\mathcal A} : X\rightarrow P(Y)$ with the number $z\in {\mathbb C},$ $z{\mathcal A}$ for short, is defined by
$D(z{\mathcal A}):=D({\mathcal A})$ and $(z{\mathcal A})(x):=z{\mathcal A}x,$ $x\in D({\mathcal A}).$ It is clear that $z{\mathcal A}  : X\rightarrow P(Y)$ is an MLO and $(\omega z){\mathcal A}=\omega(z{\mathcal A})=z(\omega {\mathcal A}),$ $z,\ \omega \in {\mathbb C}.$

Assume now that ${\mathcal A}$ is an MLO in $X.$
Then
the resolvent set of ${\mathcal A},$ $\rho({\mathcal A})$ for short, is defined as the union of those complex numbers
$\lambda \in {\mathbb C}$ for which
\begin{itemize}
\item[(i)] $X= R(\lambda-{\mathcal A})$;
\item[(ii)] $(\lambda - {\mathcal A})^{-1}$ is a single-valued bounded operator on $X.$
\end{itemize}
The operator $\lambda \mapsto (\lambda -{\mathcal A})^{-1}$ is called the resolvent of ${\mathcal A}$ ($\lambda \in \rho({\mathcal A})$); $R(\lambda : {\mathcal A})\equiv  (\lambda -{\mathcal A})^{-1}$  ($\lambda \in \rho({\mathcal A})$).
The basic properties of resolvent sets of single-valued linear operators continue to hold in our framework  (\cite{faviniyagi}, \cite{FKP}).

For the notions of various types of degenerate regularized solution operator families subgenerated by multivalued linear operators, 
we refer the reader to \cite{FKP}. 

\section{Almost periodic and asymptotically almost periodic solutions of abstract semilinear Cauchy inclusions}\label{prckovic}

Composition theorems for two-parameter Stepanov $p$-almost periodic functions have been considered in \cite[Theorem 2.2]{comp-adv}. We start this section by investigating composition theorems for Stepanov two-parameter almost periodic and asymptotically Stepanov two-parameter almost periodic functions.

The following result states that the assertion of \cite[Theorem 2.2]{comp-adv} continues to hold for the functions defined on the real semi-axis $I =[0,\infty).$ The proof of theorem is similar to that of afore-mentioned and therefore omitted.

\begin{thm}\label{vcb} 
Suppose that the following conditions hold:
\begin{itemize}
\item[(i)] $f \in APS^{p}(I \times X : X)  $ with  $p > 1, $ and there exist a number  $ r\geq \max (p, p/p -1)$ and a function $ L_{f}\in L_{S}^{r}(I) $ such that:
\begin{align}\label{vbnmp}
\| f(t,x)-f(t,y)\| \leq L_{f}(t)\|x-y\|,\quad t\in I,\ x,\ y\in X;
\end{align}
\item[(ii)] $x \in APS^{p} (I: X),$ and there exists a set $E \subseteq I$ with $m (E)= 0$ such that
$ K :=\{x(t) : t \in I \setminus E\}$
is relatively compact in $X;$ here, $m(\cdot)$ denotes the Lebesgue measure.
\end{itemize}
Then $q:=pr/p+r \in [1, p)$ and $f(\cdot, x(\cdot)) \in APS^{q}(I : X).$
\end{thm}

As observed in \cite[Remark 2.5]{ding-prim}, the condition \eqref{vbnmp} seems to be more conventional for dealing with than
the usual Lipschitz assumption. But, then we cannot consider the value $p =1$ in Theorem \ref{vcb}: this is not the case if we accept
the existence of a Lipschitz constant $L>0$ such that
\begin{align}\label{vbnmp-frim}
\| f(t,x)-f(t,y)\| \leq L\|x-y\|,\quad t\in I,\ x,\ y\in X.
\end{align}
Speaking-matter-of-factly, an insignificant modification of the proof of \cite[Theorem 2.2]{comp-adv} shows that the following result holds true:

\begin{thm}\label{vcb-prim} 
Suppose that the following conditions hold:
\begin{itemize}
\item[(i)] $f \in APS^{p}(I \times X : X)  $ with  $p \geq 1, $ $L>0$ and \eqref{vbnmp-frim} holds.
\item[(ii)] $x \in APS^{p}(I:X),$ and there exists a set $E \subseteq I$ with $m (E)= 0$ such that
$ K =\{x(t) : t \in I \setminus E\}$
is relatively compact in $ X.$
\end{itemize}
Then $f(\cdot, x(\cdot)) \in APS^{p}(I : X).$
\end{thm}

Concerning asymptotically two-parameter Stepanov $p$-almost periodic functions, we can prove the following composition principle (cf. Lemma \ref{tricky} and Lemma \ref{tricky-prim}; the use of symbol $q$ is clear from the context):

\begin{prop}\label{bibl}
Let $I =[0,\infty).$
Suppose that the following conditions hold:
\begin{itemize}
\item[(i)] $g \in APS^{p}(I \times X : X)  $ with  $p > 1, $ and there exist a number  $ r\geq \max (p, p/p -1)$ and a function $ L_{g}\in L_{S}^{r}(I:X) $ such that \eqref{vbnmp} holds with the function $f(\cdot, \cdot)  $ replaced by the function $g(\cdot, \cdot)  $ therein.
\item[(ii)] $y \in APS^{p}(I:X),$ and there exists a set $E \subseteq I$ with $m (E)= 0$ such that
$ K =\{y(t) : t \in I \setminus E\}$
is relatively compact in X.
\item[(iii)] $f(t,x)=g(t,x)+q(t,x)$ for all $t\geq 0$ and $x\in X,$ where $\hat{q}\in C_{0}([0,\infty) \times X : L^{q}([0,1]:X))$
and $q:=pr/p+r.$
\item[(iv)] $x(t)=y(t)+z(t) $ for all $t\geq 0,$ where $\hat{z}\in C_{0}([0,\infty) : L^{p}([0,1]:X)).$
\item[(v)]  There exists a set $E' \subseteq I$ with $m (E')= 0$ such that
$ K' =\{x(t) : t \in I \setminus E'\}$
is relatively compact in $ X.$
\end{itemize}
Then $q\in [1, p)$ and $f(\cdot, x(\cdot)) \in AAPS^{q}(I : X).$
\end{prop}

\begin{proof}
By Theorem \ref{vcb}, we have that the function $t\mapsto g(t, y(t)),$ $t\geq 0$ is Stepanov $q$-almost periodic. Since
\begin{align*}
f(t, x(t))=\bigl[g(t, x(t))-g(t, y(t))\bigr]+g(t, y(t))+q(t, x(t)),\quad t\geq 0,
\end{align*}
it suffices to show that
\begin{align}\label{por}
\lim_{t\rightarrow +\infty}\Biggl( \int^{t+1}_{t}\bigl\| g(s, x(s))-g(s, y(s))\bigr\|^{q} \, ds\Biggr)^{1/q}=0
\end{align}
and
\begin{align}\label{pol}
\lim_{t\rightarrow +\infty}\Biggl( \int^{t+1}_{t}\bigl\| q(s, x(s))\bigr\|^{q} \, ds\Biggr)^{1/q}=0.
\end{align}
To see that \eqref{por} holds, we can argue as in the proof of estimate \cite[(2.12)]{comp-adv}. More precisely, by \eqref{vbnmp-frim} and the H\"older inequality, we have that
\begin{align*}
\Biggl( \int^{t+1}_{t}\bigl\| g(s, x(s))&-g(s, y(s))\bigr\|^{q} \, ds\Biggr)^{1/q}
\\ & \leq \Biggl( \int^{t+1}_{t}L_{g}(s)^{q}\bigl\| x(s)-y(s)\bigr\|^{q} \, ds\Biggr)^{1/q}
\\ & \leq \Biggl( \int^{t+1}_{t}L_{g}(s)^{r}\, ds\Biggr)^{1/r} \Biggl( \int^{t+1}_{t}\bigl\| x(s)-y(s)\bigr\|^{p} \, ds\Biggr)^{1/p}
\\ & =\Biggl( \int^{t+1}_{t}L_{g}(s)^{r}\, ds\Biggr)^{1/r} \Biggl( \int^{t+1}_{t}\bigl\| z(s)\bigr\|^{p} \, ds\Biggr)^{1/p},\quad t\geq 0.
\end{align*}
Hence, \eqref{por} holds on account of $S^{r}$-boundedness of function $L_{g}(\cdot ) $ and inclusion $\hat{z}\in C_{0}([0,\infty) : L^{p}([0,1]:X)).$ The proof of \eqref{pol} follows immediately from the facts that $\hat{q}\in C_{0}([0,\infty) \times X : L^{q}([0,1]:X))$
and $ K' =\{x(t) : t \in I \setminus E'\}$
is relatively compact in $X.$
\end{proof}

If we accept the Lipschitz assumption \eqref{vbnmp-frim}, then the following result holds true:

\begin{prop}\label{bibl-frimaonji}
Let $I =[0,\infty).$
Suppose that the following conditions hold:
\begin{itemize}
\item[(i)] $g \in APS^{p}(I \times X : X)  $ with  $p \geq 1, $ and there exists a constant $L>0$ such that \eqref{vbnmp-frim} holds with the function $f(\cdot, \cdot)  $ replaced by the function $g(\cdot, \cdot)  $ therein.
\item[(ii)] $y \in APS^{p}(I:X),$ and there exists a set $E \subseteq I$ with $m (E)= 0$ such that
$ K =\{y(t) : t \in I \setminus E\}$
is compact in X.
\item[(iii)] $f(t,x)=g(t,x)+q(t,x)$ for all $t\geq 0$ and $x\in X,$ where $\hat{q}\in C_{0}([0,\infty) \times X : L^{p}([0,1]:X)).$
\item[(iv)] $x(t)=y(t)+z(t) $ for all $t\geq 0,$ where $\hat{z}\in C_{0}([0,\infty) : L^{p}([0,1]:X)).$
\item[(v)] There exists a set $E' \subseteq I$ with $m (E')= 0$ such that
$ K' =\{x(t) : t \in I \setminus E'\}$
is relatively compact in $ X.$
\end{itemize}
Then $f(\cdot, x(\cdot)) \in AAPS^{p}(I : X).$
\end{prop}

For the sequel, we need to remind ourselves of the following result recently established in \cite{EJDE}:

\begin{lem}\label{ravi-and}
Suppose that $1\leq p <\infty,$ $1/p +1/q=1$
and $(R(t))_{t> 0}\subseteq L(X)$ is a strongly continuous operator family satisfying that $M:=\sum_{k=0}^{\infty}\|R(\cdot)\|_{L^{q}[k,k+1]}<\infty .$ If $f : {\mathbb R} \rightarrow X$ is $S^{p}$-almost periodic, then the function $F(\cdot),$ given by
\begin{align}\label{wer}
F(t):=\int^{t}_{-\infty}R(t-s)f(s)\, ds,\quad t\in {\mathbb R},
\end{align}
is well-defined and almost periodic.
\end{lem}

\begin{rem}\label{conclusion}
Suppose that $t\mapsto \|R(t)\|,$ $t\in (0,1]$ is an element of the space $L^{q}[0,1].$
Then
the inequality $\sum_{k=0}^{\infty}\|R(\cdot)\|_{L^{q}[k,k+1]}<\infty $ holds provided that\\ $(R(t))_{t>0}$ is exponentially decaying at infinity or that there exists a finite number $\zeta <0$ such that $\|R(t)\| =O(t^{\zeta}),$ $t\rightarrow +\infty$ and
\begin{itemize}
\item[(i)] $p=1$ and $\zeta <-1,$ or
\item[(ii)]  $p>1$ and $\zeta <(1/p)-1.$
\end{itemize}
\end{rem}

We need to prove the following extension of \cite[Lemma 4.1]{brazil}, as well.

\begin{lem}\label{andrade}
Suppose that $(R(t))_{t>0}\subseteq L(X)$ is strongly
continuous and
$\|R(t) \| =O( e^{-\omega t}t^{\beta -1}),$ $t> 0$ for some numbers $\omega>0$ and $ \beta> 0.$ Let $f\in AAPS^{q}([0,\infty) : X)$ with some $q\in [1,\infty),$ let $1/q+1/q'=1,$ and let the following hold:
\begin{align}\label{pianino}
q'(\beta -1)>-1,\mbox{ provided }q>1\mbox{ and }\beta=1,\mbox{ provided }q=1.
\end{align}
Define 
\begin{align*}
H(t):=\int^{t}_{0}R(t-s)f(s)\, ds,\quad t\geq 0.
\end{align*}
Then $H\in AAP([0,\infty) : X).$
\end{lem}

\begin{proof}
Suppose that 
the locally $p$-integrable functions $g: {\mathbb R} \rightarrow X,$ 
$q: [0,\infty)\rightarrow X$ satisfy the conditions from Lemma \ref{tricky}. Let the function $G(\cdot)$ be given by \eqref{wer}, with  $R(\cdot)$
replaced therein by $T(\cdot);$ then we know from Lemma \ref{ravi-and} that $G(\cdot)$ is almost periodic.
Set
\begin{align*}
F(t):=\int^{t}_{0}T(t-s)q(s)\, ds-\int^{\infty}_{t}T(s)g(t-s)\, ds,\quad t\geq 0.
\end{align*}
Using H\"older inequality, we can simply prove that $H(\cdot)$ is well-defined. Since $H(t)=G(t)+F(t)$ for all $ t\geq 0,$ it suffices to show that $F\in C_{0}([0, \infty) : X).$ It is clear that
\begin{align*}
\Biggl\|\int^{\infty}_{t}T(s)& g(t-s)\, ds\Biggr\|\leq \sum_{k=0}^{\infty}\|R(\cdot)\|_{L^{q'}[t+k,t+k+1]}  \|g\|_{S^{q}}
\\ & \leq \sum_{k=0}^{\infty}\|R(\cdot)\|_{L^{\infty}[t+k,t+k+1]}  \|g\|_{S^{q}}
\leq \sum_{k=0}^{\infty}\|R(\cdot)\|_{L^{\infty}[t+k,t+k+1]}  \|g\|_{S^{q}}
\\ & \leq \mbox{Const. }\|g\|_{S^{q}}e^{-ct},\quad t> 1,
\end{align*}
so that $\lim_{t\rightarrow \infty}\int^{\infty}_{t}T(s) g(t-s)\, ds=0.$ 
Arguing as above, we get that
\begin{align*}
\Biggl\| \int^{t/2}_{0}T(t-s)& q(s)\, ds\Biggr\| \leq \|g\|_{S^{q}} \sum_{k=0}^{\lceil t/2\rceil}\|R(t-\cdot)\|_{L^{q'}[k,k+1]}
\\ & \leq M(1+\lceil t/2\rceil)e^{-c (t-\lceil t/2\rceil -1)}\|g\|_{S^{q}},\quad t\geq 2,
\end{align*}
so that $\lim_{t\rightarrow \infty}\int^{t/2}_{0}T(t-s)q(s)\, ds=0.$
Therefore, it remains to be proved that $\lim_{t\rightarrow \infty}\int^{t}_{t/2}T(t-s)q(s)\, ds=0$ (observe that the integral in this limit expression converges by \eqref{pianino} and the $S^{q}$-boundedness of function $q(\cdot)$). For that, fix a number
$\epsilon>0.$ Then there exists $t_{0}>0$ such that $\int^{t+1}_{t}\|q(s)\|^{q}\, ds<\epsilon^{q},$ $t\geq t_{0}.$
Let $t>2t_{0}+6.$ Then the H\"older inequality implies the existence of a finite constant $c>0$ such that:
\begin{align*}
\Biggl\| \int^{t}_{t/2}T(t-s)& q(s)\, ds\Biggr\| 
\\ &  \leq c\sum_{k=0}^{\lfloor t/2\rfloor-2}\|R(t-\cdot)\|_{L^{q'}[t/2+k,t/2+k+1]}\epsilon +\epsilon \bigl \|\cdot^{\beta-1}\bigr\|_{L^{q'}[0,2]}
\\ &  \leq c\sum_{k=0}^{\lfloor t/2\rfloor-2}\|R(t-\cdot)\|_{L^{\infty}[t/2+k,t/2+k+1]}\epsilon +\epsilon \bigl \|\cdot^{\beta-1}\bigr\|_{L^{q'}[0,2]}
\\ &  \leq c\epsilon M \sum_{k=0}^{\lfloor t/2\rfloor-2}e^{-c(t/2+k)}+\epsilon \bigl \|\cdot^{\beta-1}\bigr\|_{L^{q'}[0,2]}
\\ &  \leq c\epsilon M e^{-ct/2   }\sum_{k=0}^{\infty}e^{-ck}
+\epsilon \bigl \|\cdot^{\beta-1}\bigr\|_{L^{q'}[0,2]}.
\end{align*}
This yields the final conclusion.
\end{proof}

Suppose now that the condition (P) holds.
Then there exists a degenerate strongly continuous semigroup $(T(t))_{t> 0}\subseteq L(X)$ generated by ${\mathcal A}$ and
$\|T(t) \| =O( e^{-ct}t^{\beta -1}),$ $t> 0$ (\cite{EJDE}).
By a mild solution of \eqref{favini}, we mean any continuous function $u(\cdot)$ such that $u(t)= (\Lambda u)(t),$ $t\in {\mathbb R},$ where
$$
t\mapsto (\Lambda u)(t):=\int_{-\infty}^{t}T(t-s)f(s,u(s))\, ds,\ t\in {\mathbb R}.
$$

\begin{thm}\label{zeljezo}
Suppose that $f \in APS^{p}({\mathbb R} \times X : X)  $ with  $p > 1, $ and there exist a number  $ r\geq \max (p, p/p -1)$ and a function $ L_{f}\in L_{S}^{r}({\mathbb R}) $ such that \emph{\eqref{vbnmp}} holds with $I={\mathbb R}.$
Let the following condition hold:
\begin{align}\label{bez-kuce}
\beta=1,\mbox{ provided }r=p/p-1 \mbox{ and }\frac{pr}{pr-p-r}<\frac{1}{1-\beta},\mbox{ provided }r>p/p-1.
\end{align}
Set 
\begin{align*}
q':=\infty,\mbox{ provided }r=p/p-1 \mbox{ and }q':=\frac{pr}{pr-p-r},\mbox{ provided }r>p/p-1.
\end{align*}
Assume that $M:=\sum_{k=0}^{\infty}\|T(\cdot)\|_{L^{q'}[k,k+1]}<\infty $ and $M\|L_{f}\|_{S^{r}}<1.$ Then there exists a unique almost periodic mild solution of \emph{\eqref{favini}}.
\end{thm}

\begin{proof}
Since the range of any function $u \in AP ({\mathbb R} : X)$  is relatively compact in $X,$
Theorem \ref{vcb} yields that $f(\cdot, u(\cdot)) \in APS^{q}({\mathbb R} : X),$ where $q=pr/p+r.$ Since $(T(t))_{t>0}$ is exponentially decaying at infinity and $1/q' +1/q=1,$ 
the condition \eqref{bez-kuce} yields that $M<\infty .$ 
Therefore,
we can apply Lemma \ref{ravi-and} (see also Remark \ref{conclusion}) in order to see that the mapping $\Lambda : AP ({\mathbb R} : X) \rightarrow AP ({\mathbb R} : X)$ is well-defined. Furthermore, for every $t\in {\mathbb R},$ we have by H\"older inequality:
\begin{align*}
\Bigl\| & (\Lambda u)(t) -(\Lambda v)(t)   \Bigr \|=\Biggl\| \int^{\infty}_{0}T(s) \bigl[f(t-s,u(t-s))-f(t-s,v(t-s))\bigr] \, ds \Biggr \|
\\ & \leq \sum_{k=0}^{\infty} \int^{k+1}_{k}\bigl\|T(s)\bigr\| \bigl \| f(t-s,u(t-s))-f(t-s,v(t-s)) \bigr\| \, ds 
\\ &   \leq \sum_{k=0}^{\infty} \bigl\|T(\cdot)\bigr\|_{L^{q'}[k,k+1]} \bigl \|f(t-\cdot,u(t-\cdot))-f(t-\cdot,v(t-\cdot)) \bigr\|_{L^{q}[k,k+1]}
\\ & \leq \sum_{k=0}^{\infty} \bigl\|T(\cdot)\bigr\|_{L^{q'}[k,k+1]} \bigl \|L_{f}(t-\cdot) \bigl[u(t-\cdot)- v(t-\cdot)\bigr] \bigr\|_{L^{q}[k,k+1]}
\\ & \leq \sum_{k=0}^{\infty} \bigl\|T(\cdot)\bigr\|_{L^{q'}[k,k+1]}\|L_{f}\|_{S^{r}} \bigl \|u(t-\cdot)- v(t-\cdot)\bigr\|_{L^{p}[k,k+1]}
\\ & \leq \sum_{k=0}^{\infty} \bigl\|T(\cdot)\bigr\|_{L^{q'}[k,k+1]}\|L_{f}\|_{S^{r}} \bigl \| u(\cdot)- v(\cdot)\bigr\|_{L^{\infty}({\mathbb R})}.
\end{align*}
Since $M\|L_{f}\|_{S^{r}}<1,$ we can apply the Banach contraction principle to complete the proof of theorem.
\end{proof}

We can similarly prove the following result provided that the Lipschitz type condition \eqref{vbnmp-frim} holds:

\begin{thm}\label{zeljeznica}
Suppose that $f \in APS^{p}({\mathbb R} \times X : X)  $ with  $p \geq 1, $ $L>0$ and \eqref{vbnmp-frim} holds
with $I={\mathbb R}.$
Let the following condition hold:
\begin{align*}
\beta=1,\mbox{ provided }p=1 \mbox{ and }\frac{p}{p-1}<\frac{1}{1-\beta},\mbox{ provided }p>1.
\end{align*}
Set 
\begin{align*}
q':=\infty,\mbox{ provided }p=1 \mbox{ and }q':=\frac{p}{p-1},\mbox{ provided }p>1.
\end{align*}
Assume that $M:=\sum_{k=0}^{\infty}\|T(\cdot)\|_{L^{q'}[k,k+1]}<\infty $ and $ML<1.$ Then there exists a unique almost periodic mild solution of \emph{\eqref{favini}}.
\end{thm}

Let the initial value $u_{0}$ be a point of the continuity of semigroup $(T(t))_{t> 0};$ see e.g. \cite[Theorem 3.3, Theorem 3.5]{faviniyagi}. Let $ \|T(t) \| \leq M e^{-c t}t^{\beta -1},$ $t> 0$ for some constant $M> 0.$

By a mild solution $
u(\cdot)=u(\cdot;u_{0})$ of problem (DFP)$_{f,s}$ we mean any function $u\in C([0,\infty) : X)$ such that
$$
u(t)=(\Upsilon u)(t):=T(t)u_{0}+\int^{t}_{0}T(t-s)f(s,u(s))\, ds,\quad t\geq 0.
$$
Suppose that \eqref{vbnmp} holds for a.e. $t>0$ ($I=[0,\infty)$), with locally integrable positive function $L_{f}(\cdot).$ Set, for every $n\in {\mathbb N},$
\begin{align*}
M_{n}:=M^{n}&\sup_{t\geq 0}e^{-ct}\int^{t}_{0}\int^{x_{n}}_{0}\cdot \cdot \cdot \int^{x_{2}}_{0}e^{cx_{1}}
\bigl(t-x_{n}\bigr)^{\beta -1}
\\& \times \prod^{n}_{i=2}\bigl(x_{i}-x_{i-1}\bigr)^{\beta -1} \prod^{n}_{i=1}L_{f}(x_{i})\, dx_{1}\, dx_{2}\cdot \cdot \cdot \, dx_{n}.
\end{align*}
Then a simple calculation shows that
\begin{align}\label{storn}
\Bigl \| \bigl(\Upsilon^{n} u\bigr)-\bigl(\Upsilon^{n} v\bigr)\Bigr\|_{\infty}\leq M_{n}\bigl\| u- v\bigr\|_{\infty},\quad u,\ v\in BUC([0,\infty) : X),\ n\in {\mathbb N}.
\end{align}

Now we are able to state the main result of this paper:

\begin{thm}\label{stepa}
Suppose that $I=[0,\infty)$ and the following conditions hold:
\begin{itemize}
\item[(i)] $g \in APS^{p}(I \times X : X)  $ with  $p > 1, $ and there exist a number $ r\geq \max (p, p/p -1)$ and a function $ L_{g}\in L_{S}^{r}(I:X) $ such that \eqref{vbnmp} holds with the function $f(\cdot, \cdot)  $ replaced by the function $g(\cdot, \cdot)  $ therein.
\item[(ii)] $f(t,x)=g(t,x)+q(t,x)$ for all $t\geq 0$ and $x\in X,$ where $\hat{q}\in C_{0}(I \times X : L^{q}([0,1]:X))$
and $q=pr/p+r.$
\item[(iii)] 
$\beta=1,\mbox{ provided }r=p/p-1 \mbox{ and }\frac{pr}{pr-p-r}<\frac{1}{1-\beta},\mbox{ provided }r>p/p-1.$
\item[(iv)] 
\emph{\eqref{vbnmp}} holds for a.e. $t>0$, with  locally bounded positive function $L_{f}(\cdot)$ satisfying
$M_{n}<1$ for some $n\in {\mathbb N}.$
\end{itemize}
Then there exists a unique asymptotically almost periodic solution of inclusion $\emph{(DFP)}_{f,s}.$
\end{thm}

\begin{proof}
Define the number $q'$ as in the formulation of Theorem \ref{zeljezo}.
By (i)-(ii) and Proposition \ref{bibl}, we have that $f(\cdot, x(\cdot)) \in AAPS^{q}(I : X)$ for any $x\in AAP(I:X),$ where $q=pr/p+r;$ here, it is only worth observing that the range of an $X$-valued asymptotically almost periodic function is relatively compact in $X$ by \cite[Theorem 2.4]{zhang-c-prim}.
Due to (iii), the condition \eqref{pianino} holds.
Using Lemma \ref{andrade} and the obvious equality $\lim_{t\rightarrow +\infty}T(t)u_{0}=0$, we get that the mapping $
\Upsilon : AAP(X) \rightarrow AAP(X)$ is well-defined.
Making use of \eqref{storn}, (iv) and a well-known extension of the Banach contraction principle, we obtain the existence of an asymptotically almost periodic solution of inclusion $(DFP)_{f,s}.$ The uniqueness of solutions can be proved as follows: let
$u(\cdot)
$ and $
v(\cdot)
$ be two mild solutions of  inclusion $(DFP)_{f,s}.$ Then we have
\begin{align*}
\| u(t)-v(t) \| & \leq M\int^{t}_{0}e^{-c (t-s)}\bigl(t-s\bigr)^{\beta -1}L_{f}(s)\| u(s)-v(s) \|\, ds,
\quad t\geq 0.
\end{align*}
This implies by the boundedness of function $ s\mapsto e^{-c (t-s)}L(s),$ $s\in (0,t]$ 
and \cite[Lemma 6.19, p. 111]{Diet} that $ u(s)=v(s)$  for all $s\in [ 0,t]$  ($t>0$ fixed). The proof of the theorem is thereby complete.
\end{proof}

Using Proposition \ref{bibl-frimaonji} in place of Proposition \ref{bibl}, we can simply formulate and prove the following analogue of Theorem \ref{stepa} in the case of consideration of classical Lipschitz condition \eqref{vbnmp-frim}:

\begin{thm}\label{stepa-frimex}
Let $I =[0,\infty).$
Suppose that the following conditions hold:
\begin{itemize}
\item[(i)] $g \in APS^{p}(I \times X : X)  $ with  $p \geq 1, $ and there exists a constant $L>0$ such that \eqref{vbnmp-frim} holds with the function $f(\cdot, \cdot)  $ replaced by the function $g(\cdot, \cdot)  $ therein.
\item[(ii)] $f(t,x)=g(t,x)+q(t,x)$ for all $t\geq 0$ and $x\in X,$ where $\hat{q}\in C_{0}(I \times X : L^{p}([0,1]:X)).$
\item[(iii)] 
$\beta=1,\mbox{ provided }p=1 \mbox{ and }\frac{p}{p-1}<\frac{1}{1-\beta},\mbox{ provided }p>1.$
\item[(iv)] 
\emph{\eqref{vbnmp}} holds for a.e. $t>0$, with  locally bounded positive function $L_{f}(\cdot)$ satisfying
$M_{n}<1$ for some $n\in {\mathbb N}.$
\end{itemize}
Then there exists a unique asymptotically almost periodic solution of inclusion $\emph{(DFP)}_{f,s}.$
\end{thm}

Now we would like to formulate the following important consequence of Theorem \ref{stepa-frimex}:

\begin{cor}\label{kretinjo}
Suppose that $I=[0,\infty),$ the function $f(\cdot,\cdot)$ is asymptotically almost periodic and \emph{\eqref{vbnmp}} holds for a.e. $t>0$, with  locally bounded positive function $L_{f}(\cdot)$ satisfying
$M_{n}<1$ for some $n\in {\mathbb N}.$ Then there exists a unique asymptotically almost periodic solution of inclusion $\emph{(DFP)}_{f,s}.$
\end{cor}

Especially, in the case that $M_{1}<1$ in Corollary \ref{kretinjo}, we obtain the following corollary:

\begin{cor}\label{kret}
Suppose that $I=[0,\infty),$ the function $f(\cdot,\cdot)$ is asymptotically almost periodic and \emph{\eqref{vbnmp-frim}} holds for some $L\in [ 0, c^{\beta}M^{-1}\Gamma(\beta)^{-1} ).$ Then there exists a unique asymptotically almost periodic solution of inclusion $\emph{(DFP)}_{f,s}.$
\end{cor}

\begin{rem}\label{brazilci}
\begin{itemize}
\item[(i)] In the case that $\beta=1$ and $L_{f} \in L^{\infty}([0,\infty)) \cap  L^{1}([0,\infty)),$ the proof of \cite[Theorem 4.4]{brazil} shows that $\sum_{n=1}^{\infty}M_{n}<\infty,$ so that the uniqueness of solutions follows immediately by applying the Weissinger's fixed point theorem \cite[Theorem D.7]{Diet}. If the above conditions are satisfied, then the proof of Theorem \ref{stepa} can be used to state a proper extension of \cite[Theorem 4.4]{brazil}; speaking-matter-of-factly, in our approach the term $f (\cdot,u(\cdot))$ need not be 
asymptotically almost periodic and it can be of the form (iii) from the formulation of Theorem \ref{stepa}, or asymptotically Stepanov almost periodic if we consider Theorem \ref{stepa-frimex}.
Applications in the study of abstract semilinear Cauchy problems of third order:
\begin{align}\label{brafo}
\alpha u^{\prime \prime \prime}(t) + u^{\prime\prime}
(t)-\beta Au(t)-\gamma Au^{\prime}
(t) = f (t,u(t)),\quad \alpha, \ \beta, \ \gamma >  0,\ t\geq 0,
\end{align}
appearing in the theory of dynamics of elastic vibrations of flexible structures \cite{brazil}, are immediate. 
\item[(ii)] If $0<\beta<1$, then it is not trvial to state a satisfactory criterion which would enable one to see that the inequality
$M_{n}<1$ holds for some integer $n\in {\mathbb N}.$
\end{itemize}
\end{rem}

As already mentioned, it seems that the assertions of Theorem \ref{zeljezo}-Theorem \ref{stepa-frimex} are new even for 
non-degenerate semilinear differential equations with almost sectorial operators. Here we will remind ourselves of the following important result of W. von Wahl \cite{wolf}, which is most commonly used for applications
in the existing literature: 

\begin{example}\label{merkle-nedegenerisane} 
Assume that $\alpha\in (0,1),$ $m\in {\mathbb N},$
$\Omega$ is a bounded domain in ${{\mathbb R}^{n}}$ with boundary of
class $C^{4m}$ and $X:=C^{\alpha}(\overline{\Omega}).$ Define
the operator $A : D(A)\subseteq C^{\alpha}(\overline{\Omega})
\rightarrow C^{\alpha}(\overline{\Omega})$ by $D(A):=\{u\in C^{2m+\alpha}(\overline{\Omega}) :
D^{\beta}u_{|\partial \Omega}=0\mbox{ for all }|\beta|\leq m-1\}$ and
$$
Au(x):=\sum \limits_{|\beta|\leq 2m}a_{\beta}(x)D^{\beta}u(x) \mbox{
for all }x\in \overline{\Omega}.
$$
Here, $\beta \in {{\mathbb N}_{0}^{n}},$
$|\beta|=\sum_{i=1}^{n}\beta_{j},$
$D^{\beta}=\prod_{i=1}^{n}(\frac{1}{i}\frac{\partial}{\partial
x_{i}})^{\beta_{i}},$ and $a_{\beta} : \overline{\Omega} \rightarrow
{\mathbb C}$ satisfy the following conditions:
\begin{itemize}
\item[(i)] $a_{\beta}(x)\in {\mathbb R}$ for all $x\in
\overline{\Omega}$ and $|\beta|=2m.$
\item[(ii)] $a_{\beta}\in C^{\alpha}(\overline{\Omega})$ for all
$|\beta|\leq 2m,$ and
\item[(iii)] there is a constant $M>0$ such that
$$
M^{-1}|\xi|^{2m}\leq \sum
\limits_{|\beta|=2m}a_{\beta}(x)\xi^{\beta}\leq M|\xi|^{2m}\mbox{
for all }\xi \in {{\mathbb R}^{n}}\mbox{ and }x\in
\overline{\Omega}.
$$
\end{itemize}
Then there exists a sufficiently large number $\sigma>0$ such that the single-valued
operator ${\mathcal A}\equiv -(A+\sigma)$ satisfies the condition (P) with $\beta=1-\frac{\alpha}{2m}$ and some finite constants $c,\ M>0$ 
(recall that ${\mathcal A}$ is not densely defined and that the value of exponent $\beta$
in (P) is sharp). 
\end{example}

Concerning semilinear differential inclusions of first order, we would like to present the following illustrative example:

\begin{example}\label{markec} (A. Favini, A. Yagi \cite[Example 3.6]{faviniyagi}) 
Let $\Omega$ be a bounded domain in ${\mathbb R}^{n},$ $b>0,$ $m(x)\geq 0$ a.e. $x\in \Omega$, $m\in L^{\infty}(\Omega),$ $1<p<\infty$ and $X:=L^{p}(\Omega).$
Suppose that the operator $A:=\Delta -b $ acts on $X$ with the Dirichlet boundary conditions, and
that $B$ is the multiplication operator by the function $m(x).$
Then we know that the multivalued linear operator
${\mathcal A}:=AB^{-1}$ satisfies the condition (P) with $\beta=1/p$ and some finite constants $c,\ M>0;$ 
recall also that the validity of
additional condition \cite[(3.42)]{faviniyagi} on the function $m(x)$ enables us to get the better exponent $\beta$ in (P), provided that $p>2.$ Now it becomes clear how we can apply Theorem \ref{zeljezo}-Theorem \ref{zeljeznica} in the study of existence and uniqueness of almost periodic solutions of semilinear Poisson heat equation
\[\left\{
\begin{array}{l}
\frac{\partial}{\partial x}[m(x)v(t,x)]=(\Delta -b )v(t,x) +f(t,m(x)v(t,x)),\quad t\in {\mathbb R},\ x\in {\Omega};\\
v(t,x)=0,\quad (t,x)\in [0,\infty) \times \partial \Omega ,\\
\end{array}
\right.
\]
and how we can appply Theorem \ref{stepa}-Theorem \ref{stepa-frimex} in the study of existence and uniqueness of asymptotically almost periodic solutions of semilinear Poisson heat equation
\[\left\{
\begin{array}{l}
\frac{\partial}{\partial x}[m(x)v(t,x)]=(\Delta -b )v(t,x) +f(t,m(x)v(t,x)),\quad t\geq 0,\ x\in {\Omega};\\
v(t,x)=0,\quad (t,x)\in [0,\infty) \times \partial \Omega ,\\
 m(x)v(0,x)=u_{0}(x),\quad x\in {\Omega}
\end{array}
\right.
\]
in the space  $X,$ by using the substitution $u(t,x)=m(x)v(t,x)$ and passing to the corresponding semilinear differential inclusions of first order.
\end{example}

Observe, finally, that B. de Andrade and C. Lizama \cite[Theorem 4.7]{brazil} (cf. also R. Agarwal, B. de Andrade and C. Cuevas \cite[Theorem 3.5]{ravi}) 
have applied the Leray-Schauder alternative for proving the existence of asymptotically almost periodic solutions of abstract non-degenerate third order semilinear Cauchy problem \eqref{brafo}. The argumentation contained in the proof of this theorem can be applied in the analysis of existence of asymptotically almost periodic solutions for a large class of related semilinear Cauchy problems whose solutions are governed by exponentially decaying degenerate operator families that are strongly continuous for $t\geq 0$ (some examples of such operator families can be found in \cite[Chapter II]{faviniyagi}).

\end{document}